\documentclass[11pt, leqno]{amsart}
\usepackage{amsmath, amssymb, latexsym}
\usepackage{mathrsfs}
\usepackage{enumerate}
\usepackage{color}
\usepackage[colorlinks=true, pdfstartview=FitV, linkcolor=blue,%
citecolor=blue, urlcolor=blue]{hyperref}
\usepackage[all, knot]{xy}
\xyoption{arc}
\usepackage{tikz}
\usepackage{amsmath}
\numberwithin{equation}{section}

\newtheorem{theorem}{Theorem}[section]
\newtheorem{corollary}[theorem]{Corollary}

\newtheorem{proposition}[theorem]{Proposition}

\newtheorem{example}[theorem]{Example}

\theoremstyle{definition}
\newtheorem{definition}[theorem]{Definition}
\newtheorem{remark}[theorem]{Remark}

\newcommand{\Q}{\mathbb{Q}}

\newcommand{\Z}{\mathbb{Z}}

\newcommand{\h}{\mathfrak{h}}
\newcommand{\g}{\mathfrak{g}}

\title[Abstract crystals for quantum Borcherds-Bozec algebras]
{Abstract crystals for  quantum \\  Borcherds-Bozec algebras}

\author[Zhaobing Fan]{Zhaobing Fan}
\address{Harbin Engineering University,
Harbin, China}
\email{fanz@ksu.edu}
\thanks{ }

\author[Seok-Jin Kang]{Seok-Jin Kang}
\address{Korea Research Institute of Arts and Mathematics,
Asan-si, Chungcheongnam-do, 31551, Korea}
\email{soccerkang@hotmail.com}

\thanks{}

\author[Young Rock Kim]{Young Rock Kim${}^{*}$}
\address{Graduate School of Education, Hankuk University of Foreign Studies, Seoul, 02450,  Korea}
\email{rocky777@hufs.ac.kr} %
\thanks{${}^{*}$ Corresponding author. All authors contribute equally.}

\author[Bolun Tong]{Bolun Tong}
\address{Harbin Engineering University,
Harbin, China}
\email{tbl\_2019@hrbeu.edu.cn}

\address{}


\begin{document}

\begin{abstract}

In this paper, we develop the theory of abstract crystals for quantum Borcherds-Bozec algebras.
 Our construction is different from the one given by Bozec.
We further prove the crystal embedding theorem and provide a characterization of ${B}(\infty)$ and ${B}(\lambda)$  as its application,
where ${B}(\infty)$ and ${B}(\lambda)$ are  the crystals of the negative half part of the quantum Borcherds-Bozec algebra $U_q(\g)$ and its irreducible highest weight module $V(\lambda)$, respectively.

\end{abstract}

\maketitle

\setcounter{tocdepth}{1}
\tableofcontents

\section*{Introduction}

\vskip 2mm

The aim of this paper is to develop the theory of {\it abstract crystals} for quantum Borcherds-Bozec algebras.
To begin with, let us quickly review the term of  Borcherds-Bozec algebras.
The Cartan matrix of a semisimple Lie algebra is positive definite, where all of its  diagonal entries are 2.
 In \cite{Kac68, Moody68}, Kac and Moody introduced a family of infinite dimensional Lie algebras  by removing the condition of positive definitiveness,
which are now called the {\it Kac-Moody algebras}.
 In \cite{Bor88}, Borcherds generalized the notion of  Kac-Moody algebras by allowing the diagonal entries of the Cartan matrices  to be non-positive.
The Lie algebras and Cartan matrices thus obtained are called the {\it Borcherds algebras} and Borcherds-Cartan matrices, respectively.
Moreover, the Borcherds algebras are generated by positive and negative simple root vectors.
 Associated to Borcherds-Cartan matrices, the {\it Borcherds-Bozec algebras} is a further generalization of Kac-Moody algebras,
which are generated by higher degree positive and negative simple root vectors.

\vskip 2mm

 The {\it quantum  Borcherds  algebras} (or {\it quantum generalized Kac-Moody algebras}) were introduced by Kang \cite{Kang95},
 as a quantum deformation of the universal enveloping algebras of Borcherds algebras.
The {\it quantum Borcherds-Bozec algebras} were introduced by Bozec in his research of perverse sheaves theory for quivers with loops \cite{Bozec2014b, Bozec2014c, BSV2016}.
They are a further generalization of quantum Borcherds algebras.
More precisely, for each imaginary index $i\in I^{\text{im}}$, i.e., those $i\in I$ satisfying $a_{ii} \leq 0$ in the Cartan matrix,
 there are infinitely many generators $e_{il}, f_{il}$ $(l \in \Z_{>0})$,
 whose degrees are $l$ multiples of $\alpha_i$ and $-\alpha_i$.
 The commutation relations between them are rather complicated (cf. \cite{FKKT}).
 But, thanks to Bozec, there exists a set of primitive generators in quantum Borcherds-Bozec algebras with better properties and simpler commutation relations. This allows us to generalize the theories of quantum groups to quantum Borcherds-Bozec algebras.
\vskip 2mm

The canonical basis theory was first introduced by Lusztig in the ADE case in \cite{Lus90}, arising from his geometric construction of the negative part of the quantum groups, and it has been generalized to general cases by Lusztig \cite{Lus91,Lusztig}.
Meanwhile, Kashiwara constructed crystal base and global base for quantum groups associated with symmetrizable Kac-Moody algebras in an algebraic way \cite{Kas90,Kas91}. In \cite{Lus93}, Grojnowski and Lusztig proved that the global crystal base coincides with the canonical base introduced by Lusztig in \cite{Lus90}.
Since then, the crystal basis theory has become one of the most central themes in combinatorial and  geometric  representation theory of quantum groups.
In \cite{JKK2005},  Jeong, Kang and Kashiwara developed the crystal basis theory for quantum Borcherds algebras.
  Later on, together with Shin, they further introduced the notion of abstract crystals for quantum Borcherds algebras and investigated their fundamental properties  \cite{JKKS2007}.
\vskip 2mm

In \cite{Bozec2014c}, Bozec studied the crystal basis theory for quantum Borcherds-Bozec algebra $U_q(\g)$.
He defined the crystal basis $({ L} (\infty),{B}(\infty))$ for $U_q^-(\g)$ and $({ L}(\lambda),{B}(\lambda))$ for $V(\lambda)$, respectively,
 where $V(\lambda)$ is the irreducible highest weight $U_q(\g)$-module with a dominant integral  highest weight $\lambda$.
 He also constructed  {\it generalized crystals} and  their geometric realization for
 $U_{q}^{-}(\g)$ and $V(\lambda)$ based on the theory of perverse sheaves on Lusztig's and Nakajima's quiver varieties associated with quivers with loops (cf. \cite{KS1997, S2002, KKS2009, KKS2012}).

 \vskip 2mm

In this paper, we develop  the theory of abstract crystals for quantum Borcherds-Bozec algebras.
 Our construction is different from the one given in \cite{Bozec2014c}, where Bozec followed the framework of \cite{JKK2005}.
  Instead, we take the approach given in \cite{JKKS2007}, which  is simpler and more natural. We provide the abstract crystal structure on ${B}(\infty)$ and ${B}(\lambda)$ following our definition.
Moreover, we define the tensor products of abstract crystals  which are shown to be abstract crystals.   We also show that the associativity law holds for the tensor products of abstract crystals.

\vskip 2mm

 Furthermore, we prove the {\it crystal embedding theorem},
 which yields a procedure to determine the structure of the crystal ${B}(\infty)$ in term of elementary crystals.
 We provide a characterization of ${B}(\infty)$ and ${B}(\lambda)$ as an application of the crystal embedding theorem.
  The characterization of ${B}(\lambda)$ in the current paper is different from the one given by Bozec in \cite{Bozec2014c},
 where the method of Joseph in \cite{Joseph} was employed.

\vskip 2mm

 Our theory of abstract crystals can be applied to various general settings. For example, Bozec's geometric construction of the crystals $B(\infty)$ and $B(\lambda)$ can be reproduced through our approach.
  Our construction of $B(\infty)$ and $B(\lambda)$ can also be used directly to the theory of (dual) perfect bases for quantum Borcherds-Bozec algebras. Moreover, our theory will play a crucial role in the categorification of quantum Borcherds-Bozec algebras and their highest weight modules via Khovanov-Lauda-Rouquier algebras and their cyclotomic quotients. Currently, we are working on these problems as well as other applications.
We believe that our work on abstract crystals will provide a strong algebraic foundation for a wide variety of interesting developments in  combinatorial and geometric representation theory of quantum Borcherds-Bozec algebras.
\vskip 2mm

%

This paper is organized as follows.
In Section 1, we review some of the basic facts on  quantum Borcherds-Bozec algebras.
Under the assumption \eqref{eq:assumption}, the presentation of quantum Borcherds-Bozec algebras becomes  simpler. In Section 2, we define the notion of abstract crystals and investigate their fundamental properties. In Section 3, We define the tensor products of abstract crystals having all the desired properties.
In Section 4, we prove the crystal embedding theorem. As an application,
we provide a characterization of $ B(\infty)$ and $B(\lambda)$.

\vskip 2mm

\noindent\textbf{Acknowledgements.}

Z. Fan was partially supported by the NSF of China grant 11671108, the NSF of Heilongjiang Province grant JQ2020A001, and the Fundamental Research Funds for the central universities. S.-J.  Kang was supported by Hankuk University of Foreign Studies Research Fund. Y. R. Kim was supported by the Basic Science Research Program of the NRF (Korea) under grant No. 2015R1D1A1A01059643. S.-J.  Kang would like to express his sincere gratitude to Harbin Engineering University for their hospitality during his visit in July and November, 2019.

\vspace{12pt}

\section{The quantum Borcherds-Bozec algebras}

\vskip 2mm

Let $I$ be a finite or countably infinite index set.
An integer-valued matrix $A=(a_{ij})_{i,j\in I}$ is called an {\it even symmetrizable Borcherds-Cartan matrix} if it satisfies the following conditions:
\begin{itemize}
\item[(i)] $a_{ii}=2, 0, -2, -4, \ldots$,

\item[(ii)] $a_{ij} \in \Z_{\le 0}$ for $i \neq j$,

\item[(iii)] there is a diagonal matrix $D=\text{diag} (r_i \in
\Z_{>0} \mid i \in I)$ such that $DA$ is symmetric.
\end{itemize}
\vskip 2mm

Let $I^{\text{re}}:=\{ i \in I \mid a_{ii}=2 \}$, $I^{\text{im}}:= \{ i \in I \mid a_{ii} \le 0 \}$, and
$I^{\text{iso}}: =\{ i \in I \mid a_{ii}=0 \}$.
The elements of  $I^{\rm{re}}$ (resp. $I^{\text{im}}$, $I^{\text{iso}}$) are called {\it
real indices} (resp. {\it imaginary indices}, {\it isotropic indices}).

\vskip 2mm

A {\it Borcherds-Cartan datum} consists of

\begin{itemize}

\item[(a)] an even symmtrizable Borcherds-Cartan matrix,

\item[(b)] a free abelian group $P$, the {\it weight lattice},

\item[(c)] $\Pi = \{\alpha_{i}  \in P \mid i \in I  \}$, the set of {\it simple roots},

\item[(d)] $P^{\vee} := Hom_{\Z} (P, \,  \Z)$, the {\it dual weight lattice},

\item[(e)] $\Pi^{\vee} = \{ h_i  \mid i \in I \}$, the set of {\it simple coroots}

\end{itemize}

satisfying the following conditions:

\begin{itemize}

\item[(i)] $\langle h_i, \alpha_j  \rangle = a_{ij}$ for all $i, j \in I$,

\item[(ii)] $\Pi$ is linearly independent over $\Z$,

\item[(iii)] for each $i \in I$, there is an element $\Lambda_{i} \in P$, called the {\it fundamental weights},
such that
$$\langle h_{j}, \Lambda_{i} \rangle =\delta_{ij} \quad \text{for all} \ j \in I.$$

\end{itemize}

We denote by
$$P^+:=\{\lambda \in P \mid \langle h_i, \lambda \rangle \geq 0 \ \ \text{for all} \ i \in I \},$$
the set of {\it dominant integral weights}. The free abelian group $Q:=\bigoplus_{i \in I} {\Z \alpha_i}$ is called the {\it root lattice}. Set $Q^+=\sum_{i \in I}{\Z_{\geq0}\alpha_i}$ and $Q^-=-Q^+$. For $\beta=\sum k_i\alpha_i \in Q^+ $, we define its {\it height} to be $ \text{ht} (\beta):=\sum k_i$.

\vskip 2mm

Set  $\h = \Q \otimes_{\Z} P^{\vee}$, the {\it Cartan subalgebra}. Since $A$ is symmetrizable and
$\Pi$ is linearly independent, there is a non-degenerate symmetric bilinear form $( \  , \ )$ on $\h^*$ satisfying
$$(\alpha_i,\lambda)=r_i \langle h_i, \lambda  \rangle \ \ \text{for all} \  i \in I, \, \lambda \in \h^*.$$

For $i \in I^{\text{re}}$, we define the {\it simple reflection}
$\omega_{i} \in GL(\h^{*})$ by
$$\omega_{i} (\lambda) =  \lambda- \langle h_i, \lambda \rangle \alpha_{i}
\ \ \text{for} \ \lambda \in \h^{*}.$$
The subgroup $W$ of $GL(\h^*)$ generated by $\omega_{i}$ $(i \in I^{\text{re}})$ is called
the {\it Weyl group}. 
 One can easily verify that the
symmetric bilinear form $(\ , \ )$ is $W$-invariant.

\vskip 2mm

Let $I^{\infty}:= (I^{\text{re}} \times \{1\}) \cup (I^{\text{im}}
\times \Z_{>0})$. For simplicity, we shall often write $i$ instead of $(i,1)$ for $i \in I^{\text{re}}$.
Let $q$ be an indeterminate and set
$$q_i=q^{r_i}\ {\rm and}\  q_{(i)}=q^{\frac{(\alpha_i,\alpha_i)}{2}}.$$
Note that $q_i= q_{(i)}$ if $i \in I^{\text {re}}$. For each $i \in I^{\text {re}}$ and $n \in \Z_{\geq 0}$, we define
$$[n]_i=\frac{q_{i}^n-q_{i}^{-n}}{q_{i}-q_{i}^{-1}},\quad [n]_i!=\prod_{k=1}^n [k]_i\quad {\rm and}\quad
{\begin{bmatrix} n \\ k \end{bmatrix}}_i=\frac{[n]_i!}{[k]_i![n-k]_i!}.$$

\vskip 2mm

Let $\mathscr F$ be the free associative algebra over $\Q(q)$ generated by the symbols $f_{il}$ for $(i,l)\in I^{\infty}$.
By setting $\deg f_{il}= -l \alpha_{i} $, $\mathscr F$ becomes a $Q^-$-graded algebra.
For any homogeneous element $u$ in $\mathscr F$, we denote by $|u|$ the degree of $u$, and for any $A \subseteq Q^{-}$, we set
${\mathscr F}_{A}=\{ x\in {\mathscr F} \mid |x| \in A \}$.

\vskip 2mm

We define a {\it  twisted} multiplication on $\mathscr F\otimes\mathscr F$ by
$$(x_1\otimes x_2)(y_1\otimes y_2)=q^{-(|x_2|,|y_1|)}x_1y_1\otimes x_2y_2,$$
and a comultiplication $\delta\colon \mathscr F\rightarrow \mathscr F\otimes\mathscr F$ by
$$\delta(f_{il})=\sum_{m+n=l}q_{(i)}^{-mn}f_{im}\otimes f_{in} \ \ \text{for}  \ (i,l)\in I^{\infty}.$$
Here $f_{i0}=1$ and $f_{il}=0$ for $l<0$ by convention.

\vskip 2mm

\begin{proposition}\cite{Bozec2014b,Bozec2014c} \
{\rm For any family $\nu=(\nu_{il})_{(i,l)\in I^{\infty}}$ of non-zero elements in $\Q(q)$, there exists a symmetric bilinear form $( \ , \ )_L \colon \mathscr F\times\mathscr F\rightarrow \Q(q)$ such that

\begin{itemize}
\item[(a)] $(x, y)_{L} =0$ if $|x| \neq |y|$.

\item[(b)] $(1,1)_{L} = 1$.

\item[(c)] $(f_{il}, f_{il})_{L} = \nu_{il}$ for all $(i,l) \in
I^{\infty}$.

\item[(d)] $(x, yz)_{L} = (\delta(x), y \otimes z)_L$  for all $x,y,z
\in {\mathscr F}$,
\end{itemize}
 where $(x_1\otimes x_2,y_1\otimes y_2)_L=(x_1,y_1)_L(x_2,y_2)_L$ for any $x_1,x_2,y_1,y_2\in {\mathscr F}$.}
\end{proposition}

From now on, we will assume that
\begin{equation} \label{eq:assumption}
\nu_{il} \in 1+q\Z_{\geq0}[[q]]\ \ \text{for all} \ (i,l)\in I^{\infty}.
\end{equation}
Under the assumption \eqref{eq:assumption}, the radical of the bilinear form $( \ , \ )_L$ is generated by
\begin{equation}\label{re}
\begin{aligned}
& \sum_{k=0}^{1-la_{ij}}(-1)^k
{\begin{bmatrix} 1-la_{ij} \\ k \end{bmatrix}}_i{f_i}^{{1-la_{ij}}-k}f_{jl}f_i^k=0 \ \ \text{for} \ i\in
I^{\text{re}}, \, i \neq (j,l), \vspace{5pt} \\
& f_{ik}f_{jl}-f_{jl}f_{ik} =0 \ \ \text{for} \ a_{ij}=0.
\end{aligned}
\end{equation}

\vskip 2mm

Let $U^{\leq0}$ be the associative algebra over $\Q(q)$ with $\mathbf 1$ generated by $f_{il}$ $((i,l) \in I^{\infty})$ and $q^h$ $(h\in P^{\vee})$ with the defining relations \eqref{re} and
\begin{equation}
\begin{aligned}
& q^0=\mathbf 1,\quad q^hq^{h'}=q^{h+h'} \ \ \text{for} \ h,h' \in P^{\vee}, \\
& q^h f_{jl}q^{-h} = q^{-l \langle h, \alpha_j \rangle} f_{jl}\ \ \text{for} \ h \in P^{\vee}, (j,l)\in I^{\infty}. \\
\end{aligned}
\end{equation}

The comultiplication $\delta$ induces a well-defined comultiplication
$\Delta\colon U^{\leq0}\rightarrow U^{\leq0}\otimes U^{\leq0}$ given by
\begin{equation} \label{eq:comult}
\begin{aligned}
& \Delta(q^h) = q^h \otimes q^h, \\
& \Delta(f_{il}) = \sum_{m+n=l} q_{(i)}^{-mn}f_{im}K_{i}^{n}\otimes f_{in} .
\end{aligned}
\end{equation}

\vskip 2mm

By the Drinfeld double process, we have the following definition for quantum Borcherds-Bozec algebra.
\begin{definition}
The {\it quantum Borcherds-Bozec algebra} $U_q(\g)$ associated with a Borcherds-Cartan datum $(A, P, \Pi, P^{\vee}, \Pi^{\vee})$
is the associative algebra over $\Q(q)$ with $\mathbf 1$
generated by the elements $q^h$ $(h\in P^{\vee})$ and $e_{il},
f_{il}$ $((i,l) \in I^{\infty})$ with the following defining relations:
\begin{equation} \label{eq:rels}
\begin{aligned}
& q^0=\mathbf 1,\quad q^hq^{h'}=q^{h+h'} \ \ \text{for} \ h,h' \in P^{\vee}, \\
& q^h e_{jl}q^{-h} = q^{l \langle h, \alpha_j \rangle} e_{jl}, \ \
q^h f_{jl}q^{-h} = q^{-l \langle h, \alpha_j \rangle} f_{jl}\ \ \text{for} \ h \in P^{\vee}, (j,l)\in I^{\infty}, \\
& e_{ik}f_{jl}-f_{jl}e_{ik}=0 \ \ \text{for} \ i\neq j, \\
& \sum_{\substack{m+n=k \\ n+s=l}}q_{(i)}^{n(m-s)}\nu_{in}e_{is}f_{im}K_i^{-n}=\sum_{\substack{m+n=k \\ n+s=l}}q_{(i)}^{-n(m-s)}\nu_{in}f_{im}e_{is}K_i^{n}\ \ \text{for} \ (i,l),(i,k)\in I^{\infty},\\
& \sum_{k=0}^{1-la_{ij}}(-1)^k
{\begin{bmatrix} 1-la_{ij} \\ k \end{bmatrix}}_i{e_i}^{{1-la_{ij}}-k}e_{jl}e_i^k=0 \ \ \text{for} \ i\in
I^{\text{re}}, \, i \neq (j,l), \\
& \sum_{k=0}^{1-la_{ij}}(-1)^k
{\begin{bmatrix} 1-la_{ij} \\ k \end{bmatrix}}_i{f_i}^{{1-la_{ij}}-k}f_{jl}f_i^k=0 \ \ \text{for} \ i\in
I^{\text{re}}, \, i \neq (j,l), \\
& e_{ik}e_{jl}-e_{jl}e_{ik} = f_{ik}f_{jl}-f_{jl}f_{ik} =0 \ \ \text{for} \ a_{ij}=0.
\end{aligned}
\end{equation}
Here $K_i=q_{i}^{h_i}$ $(i \in I)$. We extend the grading by setting $|q^h|=0$ and $|e_{il}|= l \alpha_{i}$.
\vskip 3mm
\end{definition}

\begin{remark}
Our presentation of $U_q(\g)$ is simpler than the ones in \cite{Bozec2014b, Bozec2014c, BSV2016} thanks to the explicit commutation relations among the generators given in \cite[Appendix A]{FKKT}.
\end{remark}

\vskip 2mm

The comultiplication on $U^{\le 0}$ can be extended to  the comultiplication
$\Delta\colon U_q(\g)\rightarrow U_q(\g)\otimes U_q(\g)$ given by
\begin{equation} \label{eq:comult}
\begin{aligned}
& \Delta(q^h) = q^h \otimes q^h, \\
& \Delta(e_{il}) = \sum_{m+n=l} q_{(i)}^{mn}e_{im}\otimes K_{i}^{-m}e_{in}, \\
& \Delta(f_{il}) = \sum_{m+n=l} q_{(i)}^{-mn}f_{im}K_{i}^{n}\otimes f_{in}.
\end{aligned}
\end{equation}

\vskip 2mm

Let $U^+$ (resp. $U^-$) be the subalgebra of $U_q(\g)$ generated by $e_{il}$ (resp. $f_{il}$) for $(i,l)\in I^{\infty}$,
and $U^{0}$ the subalgebra of $U_q(\g)$ generated by $q^h$ for $h\in P^{\vee}$.
Then the quantum Borcherds-Bozec algebra $U_q(\g)$ has the following {\it triangular decomposition}
$$U_q{(\g)}\cong U^-\otimes U^0 \otimes U^+.$$

Let $\omega\colon U_q(\g)\rightarrow U_q(\g)$  be the involution defined by
$$\omega(q^h)=q^{-h},\ \omega(e_{il})=f_{il},\ \omega(f_{il})=e_{il}\ \ \text{for}\ h \in P^{\vee},\ (i,l)\in I^{\infty},$$
and define a $\Q$-algebra involution on $U_q(\g)$, called the {\it bar involution},  by
$$\overline{e}_{il}=e_{il},\ \overline{f}_{il}=f_{il},\ \overline{q}^h=q^{-h},\ \overline{q}=q^{-1}\ \ \text{for} \ h\in P^{\vee},\ (i,l)\in I^{\infty}.$$

\vskip 2mm

The following proposition provides a set of {\it primitive generators} for $U_q(\g)$.

\vskip 2mm

\begin{proposition}\cite{Bozec2014b,Bozec2014c}\label{prim}
{\rm For any $i\in I^{\text {im}}$ and $l\geq 1$, there exist unique elements $t_{il}\in  U^-_{-l \alpha_{i}}$ and $s_{il}=\omega (t_{il})$ such that
\begin{itemize}\label{bozec}
\item[(1)] $\Q (q) \left<f_{il} \mid l\geq 1\right>=\Q (q) \left<t_{il} \mid l\geq 1\right>$ and $\Q (q) \left<e_{il} \mid l\geq 1\right>=\Q (q) \left<s_{il} \mid l\geq 1\right>$,
\item[(2)] $(t_{il},z)_L=0$ for all $z\in \Q (q) \left<f_{i1} ,\cdots,f_{i,l-1}\right>$,\\
$(s_{il},z)_L=0$ for all $z\in \Q (q) \left<e_{i1} ,\cdots,e_{i,l-1}\right>$,
\item[(3)] $t_{il}-f_{il}\in \Q (q) \left<f_{ik} \mid k<l \right>$ and $s_{il}-e_{il}\in \Q (q) \left<e_{ik} \mid k<l \right>$,
\item[(4)] $\overline{t}_{il}=t_{il},\ \ \overline{s}_{il}=s_{il}$,
\item[(5)] $\delta(t_{il})=t_{il}\otimes 1+1\otimes t_{il}, \ \delta(s_{il})=s_{il}\otimes 1+ 1\otimes s_{il}$,
\item[(6)] $\Delta(t_{il})=t_{il}\otimes 1+K_i^l\otimes t_{il}, \ \Delta(s_{il})=s_{il}\otimes K_i^{-l}+ 1\otimes s_{il}$,
\end{itemize}}
\end{proposition}



\vskip 2mm

Set $\tau_{il}=(t_{il},t_{il})_L=(s_{il},s_{il})_L$. We have the following commutation relations in $U_q(\g)$
\begin{equation}\label{news}
s_{il} \, t_{jk}-t_{jk} \,s_{il}=\delta_{ij}\delta_{lk}\tau_{il}(K_i^l-K_i^{-l}).
\end{equation}
For $i\in I^{\text{im}} \backslash I^{\text{iso}}$ (resp. $i\in I^{\text{iso}}$), denote by $\mathcal C_{i,l}$ the set of compositions (resp. partitions) of $l$, and set $\mathcal C_i=\bigsqcup_{l\geq 0}\mathcal C_{i,l}$. For $i\in I^{\text{re}}$, we put $\mathcal C_{i,l}=\{ (l) \}$.

\vskip 2mm

If $i\in I^{\text {im}}$ and $\mathbf c =(c_1,\cdots,c_r)\in \mathcal C_{i,l}$, we write $|\mathbf c|=l$ and define
  $$t_{i,\mathbf c}=t_{ic_1}\cdots t_{ic_r},\quad  s_{i,\mathbf c}=s_{ic_1}\cdots s_{ic_r}\ \ {\rm and}\ \ \tau_{i,\mathbf c}=\tau_{ic_1}\cdots \tau_{ic_r}.$$
  Note that $\{ t_{i,\mathbf c} \mid \mathbf c \in \mathcal C_{i,l}\}$
 forms a basis of  $U^-_{-l \alpha_{i}}$.

\vskip 2mm

For each $(i, l) \in I^{\infty}$, we define the linear map $e_{il}'\colon U^{-} \rightarrow U^{-}$ by
\begin{equation}\label{delta}
e_{il}'(1)=0, \ e_{il}'(t_{jk})=\delta_{ij}\delta_{lk}\ \text{and}\ e_{il}'(xy)=e_{il}'(x)y+q^{l(|x|,\alpha_i)}xe_{il}'(y).
\end{equation}
In \cite{Bozec2014c}, Bozec used the operators $e_{il}'$ to define the {\it Kashiwara operators} on $U^{-}$.


\vspace{12pt}

\section{Abstract crystals}

\vskip 2mm

In this section, we develop the theory of abstract crystals for quantum Borcherds-Bozec algebras. We would like to point out that we follow the outline given in \cite{JKKS2007}, not in \cite{JKK2005}.

\begin{definition}\label{ac}
Let $U_q(\g)$ be the quantum Borcherds-Bozec algebra associated with a given Borcherds-Cartan datum $(A, P, P^{\vee}, \Pi, \Pi^{\vee})$.
An {\it abstract $U_q(\g)$-crystal} or simply a {\it crystal} is a set $B$ together with the maps ${\rm wt}\colon  B \rightarrow P$, $\widetilde{e}_{il},\widetilde{f}_{il}\colon B\rightarrow  B\sqcup \{0\}$ $((i,l)\in I^{\infty})$ and $\epsilon_i,\phi_i\colon B\rightarrow \Z\sqcup \{-\infty\}$ $(i\in I)$ satisfying the following conditions:
\begin{itemize}
\item [(i)] $\text{wt}(\widetilde{e}_{i,l}b)=\text{wt}(b)+l\alpha_i$ if $\widetilde{e}_{il}b\neq 0$.
\item [(ii)] $\text{wt}(\widetilde{f}_{i,l}b)=\text{wt}(b)-l\alpha_i$ if $\widetilde{f}_{il}b\neq 0$.
\item [(iii)] For any $i\in I$ and $b\in  B$, $\phi_i(b)=\epsilon_i(b)+ \langle h_i, \text{wt}(b) \rangle$.
\item [(iv)] For any $(i,l)\in I^{\infty}$ and $b,b'\in  B$, $\widetilde{f}_{il}b=b'$ if and only if $b=\widetilde{e}_{il}b'$.
\item [(v)] For any $i\in I^{\text{re}}$ and $b\in  B$, we have
\begin{itemize}
\item [(a)] $\epsilon_i(\widetilde{e}_ib)=\epsilon_i(b)-1$, $\phi_i(\widetilde{e}_ib)=\phi_i(b)+1$ if $\widetilde{e}_ib\neq 0$,
\item [(b)] $\epsilon_i(\widetilde{f}_ib)=\epsilon_i(b)+1$, $\phi_i(\widetilde{f}_ib)=\phi_i(b)-1$ if $\widetilde{f}_ib\neq 0$,
\end{itemize}
where $\widetilde{e}_i=\widetilde{e}_{i1}$ and $\widetilde{f}_i=\widetilde{f}_{i1}$.
\item [(vi)] For any $i\in I^{\text{im}}$, $l>0$ and $b\in B$, we have
\begin{itemize}
\item [(a)] $\epsilon_i(\widetilde{e}_{il}b)=\epsilon_i(b)$, $\phi_i(\widetilde{e}_{il}b)=\phi_i(b)+la_{ii}$ if $\widetilde{e}_{il}b\neq 0$,
\item [(b)] $\epsilon_i(\widetilde{f}_{il}b)=\epsilon_i(b)$, $\phi_i(\widetilde{f}_{il}b)=\phi_i(b)-la_{ii}$ if $\widetilde{f}_{il}b\neq 0$.
\end{itemize}
\item [(vii)] For any $(i,l)\in I^{\infty}$ and $b\in B$ such that $\phi_i(b)=-\infty$, we have $\widetilde{e}_{il}b=\widetilde{f}_{il}b=0$.
\end{itemize}
\end{definition}

\vskip 2mm

Let $B$ be an abstract crystal. For $b, b' \in B$ and $(i, l) \in I^{\infty}$, by the condition (iv), we have $\widetilde{f}_{il}  b = b'$ if and only if $\widetilde{e}_{il} b' = b$. In this case,we draw a labelled arrow $b \overset{(i,l)} \longrightarrow b'$. The directed graph thus obtained is called the {\it crystal graph} of $B$.

\vskip 2mm


\begin{definition}\label{mor}

Let $B_1$ and $B_2$ be abstract crystals. A map $\psi\colon B_1 \rightarrow B_2$ is called a {\it morphism of crystals} or a {\it crystal morphism} if it satisfies the following conditions:
\begin{itemize}
\item [(i)] for $b\in B_1$ and $i\in I$, we have\\
$\text{wt}(\psi(b))=\text{wt}(b)$,  $\epsilon_i(\psi(b))=\epsilon_i(b)$, $\phi_i(\psi(b))=\phi_i(b)$,
\item [(ii)] for all $b\in B_1$,  $(i,l)\in I^{\infty}$, if $\widetilde{f}_{il}b\in B_1$, then $\psi(\widetilde{f}_{il}b)=\widetilde{f}_{il}\psi(b)$.
\end{itemize}
\end{definition}

\begin{remark} \label{rem:iso} \hfill

(a) If $b\in B_1$ and $\widetilde{e}_{il}b\in B_1$, one can deduce $\psi(\widetilde{e}_{il}b)=\widetilde{e}_{il}\psi(b)$.

(b) If a crystal morphism $\psi$ is a bijection, then $\psi^{-1}$ is also a crystal morphism. In particular, both $\psi$ and $\psi^{-1}$ commute with $\widetilde{e}_{il}$'s. 

\end{remark}


\begin{definition} \label{def:morphism}

Let $\psi\colon B_1\rightarrow B_2$ be a morphism of crystals.
\begin{itemize}
\item [(a)] $\psi$ is called a {\it strict morphism} if
$$\psi(\widetilde{e}_{il} b) = \widetilde{e}_{il}(\psi(b)), \ \  \psi(\widetilde{f}_{il} b) = \widetilde{f}_{il}(\psi(b))$$
for all $b \in B_1$, $(i,l)\in I^{\infty}$. Here, we understand $\psi(0)=0$.

\item [(b)] $\psi$ is called a {\it crystal embedding} if the underlying map $\psi\colon B_1\rightarrow B_2$ is injective.
In this case, $B_1$ is called a {\it subcrystal} of $B_2$.
If $\psi$ is a strict embedding, $B_1$ is called  a {\it full subcrystal} of $B_2$.

\item[(c)] $\psi$ is called an {\it isomorphism} if it is a bijection.

\end{itemize}
\end{definition}

\vskip 2mm

In the following two examples, we recall the crystals $B(\infty)$ and $B(\lambda)$ $(\lambda \in P^{+})$ constructed in \cite{Bozec2014c}.

\begin{example} \label{ex:Binfty} \hfill

{\rm  Let $u \in U^{-}$ and $(i,l) \in I^{\infty}$.

\vskip 2mm

If $i \in I^{\text{re}}$, in \cite{Kas91}, Kashiwara proved that $u$ can be uniquely written as
$$u = \sum_{k \ge 0} f_{i}^{(k)} u_k,$$
where $f_{i}^{(k)} = f_{i}^k / [k]_{i}!$ and $e_{i}' u_{k}=0$ for all $k\ge 0$.
In this case, the {\it Kashiwara operators} are defined by
$$\widetilde{e_i} \, u = \sum_{k\ge 1} f_{i}^{(k-1)}\, u_k, \ \  \widetilde{f_i} \, u = \sum_{k\ge 0} f_{i}^{(k+1)}\, u_k.$$

\vskip 2mm

If $i \in I^{\text{im}} \setminus I^{\text{iso}}$, in \cite{Bozec2014c}, Bozec showed that $u$ can be written uniquely as
$$u = \sum_{\mathbf{c} \in \mathcal{C}_{i}} t_{i, \mathbf{c}} u_{\mathbf{c}},$$
where $\mathbf{c} = (c_1, \ldots, c_r)$ is a composition in $\mathcal{C}_{i}$ and $e_{il}'\, u_{\mathbf{c}} = 0$ for all $l\ge 1$. In this case, the Kashiwara operators are defined by
\begin{equation*}
\widetilde{e}_{il} \, u = \sum_{\mathbf{c} : \, c_1 = l} \, t_{i, \mathbf{c} \setminus c_1}\, u_{\mathbf{c}},
\ \ \widetilde{f}_{il} \, u = \sum_{\mathbf{c}\in\mathcal C_i} \, t_{i, (l, \mathbf{c})}\, u_{\mathbf{c}}.
\end{equation*}

\vskip 2mm

If $i \in I^{\text{iso}}$, $u$ can be written uniquely as
$$u = \sum_{\mathbf{c} \in \mathcal{C}_{i}} t_{i, \mathbf{c}} u_{\mathbf{c}},$$
where $\mathbf{c} = (c_1, \ldots, c_r)$ is a partition in $\mathcal{C}_{i}$ and $e_{il}'\, u_{\mathbf{c}} = 0$ for all $l\ge 1$. In this case, the Kashiwara operators are defined by
\begin{equation*}
\widetilde{e}_{il} \, u = \sum_{\mathbf{c} \in \mathcal{C}_{i}} \,\sqrt{\frac{m_{l}(\mathbf{c})}{l}} t_{i, \mathbf{c} \setminus l}\, u_{\mathbf{c}},
\ \ \ \widetilde{f}_{il} \, u = \sum_{\mathbf{c} \in \mathcal{C}_{i}} \,\sqrt{\frac{l}{m_{l}(\mathbf{c})+1}}  t_{i, {\mathbf{c} \cup l}}\, u_{\mathbf{c}},
\end{equation*}
where $m_{l}(\mathbf{c}) = \# \{k \mid c_k = l\}$, and $\mathbf{c} \cup l$ stands for the partition $(l,c_1,\cdots,c_r)$.

\vskip 2mm
\begin{remark}
Note that the square roots appear in the above definition. So we need to consider an extension $\mathbb{F}$ of $\Q$ that contains all the necessary square roots in Example \ref{ex:Binfty} and in Example \ref{ex:Blambda}. (See \cite[Remark 3.12]{Bozec2014c}.)
\end{remark}

\vskip 2mm

%

Let $\mathbb A_0=\{f\in \mathbb{F}(q) \mid f \ \text{is regular at} \ q=0\}$ and let
$L(\infty)$ be the $\mathbb A_0$-submodule of $U^-$ generated by the elements of the form
$\widetilde{f}_{i_1,l_1}\cdots\widetilde{f}_{i_r,l_r}  \mathbf{1}$
for  $r\geq 0$ and $(i_k,l_k)\in I^{\infty}$. Set
$${B}(\infty)=\{\widetilde{f}_{i_1,l_1}\cdots\widetilde{f}_{i_r,l_r} \mathbf{1} \mod{q L(\infty)} \mid r \ge 0, (i_k, l_k) \in I^{\infty}  \}\subseteq
{L}(\infty)/ q{L}(\infty).$$

For $b=\widetilde{f}_{i_1,l_1}\cdots\widetilde{f}_{i_r,l_r} \mathbf{1}$, we define
\begin{equation}
\begin{aligned}
&\text{wt}(b)=-(l_1\alpha_{i_1}+\cdots+l_r\alpha_{i_r}), \\
&\epsilon_i(b)=
\begin{cases}
\text{max}\{k\geq0\mid \widetilde{e}_i^kb\neq 0\}  &\text{for} \ i\in I^{\text{re}},\\
 0 & \text{for} \ i\in I^{\text{im}},
\end{cases}\\
&\phi_i(b)=\epsilon_i(b)+ \langle h_i, \text{wt}(b) \rangle \quad \text{for any} \ \  i\in I.
\end{aligned}
\end{equation}
Then $B(\infty)$ becomes an abstract crystal.

\vskip 2mm

Note that $\mathbf{1}$ is the only element in $B(\infty)$ of weight $0$ and is annihilated by all $\widetilde{e}_{il}$ for
$(i,l) \in I^{\infty}$. Moreover, $B(\infty)$ is connected (as a directed graph).

}
\end{example}

\vskip 2mm

\begin{example} \label{ex:Blambda}

{\rm For a dominant integral weight $\lambda \in P^{+}$, let $V(\lambda) = U_{q}(\g)  v_{\lambda}$ be the irreducible highest weight module with highest weight $\lambda$ and highest weight vector $v_{\lambda}$.

\vskip 2mm

Let $v \in V(\lambda)$.
If $i \in I^{\text{re}}$,  $v$ can be uniquely written as
$$v = \sum_{k \ge 0} f_{i}^{(k)} v_k,$$
where $f_{i}^{(k)} = f_{i}^k / [k]_{i}!$ and $e_{i} v_{k}=0$ for all $k\ge 0$, and if $i \in I^{\text{im}}$,
$v$ can be written uniquely as
$$v = \sum_{\mathbf{c} \in \mathcal{C}_{i}} t_{i, \mathbf{c}} v_{\mathbf{c}},$$
where $\mathbf{c} = (c_1, \ldots, c_r)$ is a composition or a partition  in $\mathcal{C}_{i}$ and $e_{il}\, v_{\mathbf{c}} = 0$ for all $l\ge 1$. The Kashiwara operators are defined in a similar manner as in Example \ref{ex:Binfty}.

\vskip 2mm

Let $L(\lambda)$ be the $\mathbb A_0$-submodule of $V(\lambda)$ generated by the elements of the form
$\widetilde{f}_{i_1,l_1}\cdots\widetilde{f}_{i_r,l_r} v_{\lambda}$
for  $r\geq 0$ and $(i_k,l_k)\in I^{\infty}$. Set
$${B}(\lambda)=\{\widetilde{f}_{i_1,l_1}\cdots\widetilde{f}_{i_r,l_r} v_{\lambda}  \mod{q L(\lambda)} \mid r \ge 0, (i_k, l_k) \in I^{\infty}  \} \setminus \{0\} \subseteq
{L}(\lambda)/ q{L}(\lambda).$$

For $b=\widetilde{f}_{i_1,l_1}\cdots\widetilde{f}_{i_r,l_r} v_{\lambda}$, we define
\begin{equation}
\begin{aligned}
&\text{wt}(b)=\lambda - (l_1\alpha_{i_1}+\cdots+l_r\alpha_{i_r}), \\
&\epsilon_i(b)=
\begin{cases}
\text{max}\{k\geq0\mid \widetilde{e}_i^kb\neq 0\}  & \text{for} \ i\in I^{\text{re}},\\
 0 & \text{for} \ i\in I^{\text{im}},
\end{cases}\\
&\phi_i(b)=\epsilon_i(b)+ \langle h_i, \text{wt}(b) \rangle \quad \text{for any} \ \  i\in I.
\end{aligned}
\end{equation}
Then $B(\lambda)$ becomes an abstract crystal. Note that $v_{\lambda}$ is the only element in $B(\lambda)$ of weight $\lambda$ and is annihilated by all $\widetilde{e}_{il}$ for $(i,l) \in I^{\infty}$. Moreover, $B(\lambda)$ is connected.

}
\end{example}

\vskip 2mm

\begin{example} \label{ex:Tlambda} \hfill

\vskip 2mm

{\rm

(a) For $\lambda \in P$, set $T_{\lambda} = \{ t_{\lambda} \}$ and define
\begin{equation*}
\text{wt} (t_{\lambda}) = \lambda, \ \ \epsilon_{i}(t_{\lambda}) = \phi_{i}(t_{\lambda}) = - \infty,
\ \  \widetilde{e}_{il} (t_{\lambda}) = \widetilde{f}_{il} (t_{\lambda}) =0
\end{equation*}
for all $i \in I$ and $l \ge 1$. Then $T_{\lambda}$ is an abstract crystal.

\vskip 2mm

(b) Let $C=\{c\}$ and define
\begin{equation*}
\text{wt} (c) = 0, \ \ \epsilon_{i}(c) = \phi_{i}(c) = 0,
\ \  \widetilde{e}_{il} (c) = \widetilde{f}_{il} (c) =0
\end{equation*}
for all $i \in I$ and $l \ge 1$. Then $C$ is an abstract crystal which is isomorphic to $B(0)$.

}
\end{example}

\vskip 2mm

As was proved in \cite{JKK2005}, we have the following proposition which describes the relation between the crystals $B(\infty)$ and $B(\lambda)$.

\vskip 2mm

\begin{proposition}\label{P1}
{\rm For every $\lambda \in P^+$, there exists an injective map $\pi_\lambda\colon {B} (\lambda)\rightarrow{B}(\infty)$ such that
\begin{itemize}
\item [(i)] $\pi_\lambda(v_\lambda)=1$,
\item [(ii)] for all $(i,l)\in I^{\infty}$ and $b\in {B}(\lambda)$ with $\widetilde{f}_{il}(b)\neq 0$, $\pi_\lambda \, \widetilde{f}_{il}(b)=\widetilde{f}_{il}\, \pi_\lambda(b),$
\item [(iii)] for all $(i,l)\in I^{\infty}$ and $b\in {B}(\lambda)$,
$\pi_\lambda\, \widetilde{e}_{il}(b)=\widetilde{e}_{il}\, \pi_\lambda(b),$
\item [(iv)] for all $i\in I$ and $b\in {B}(\lambda)$,
$\text{wt}(\pi_\lambda(b))=\text{wt}(b)-\lambda, \ \  \epsilon_i(\pi_\lambda(b))=\epsilon_i(b).$
\end{itemize}
}
\end{proposition}

\begin{remark} The map $\pi_{\lambda}$ is {\it not} a crystal embedding because it does not preserve the functions
$\text{wt}$ and $\phi_i$ $(i \in I)$.

\end{remark}

\vskip 2mm

The following example provides a very important class of abstract crystals $B_{i}$ $(i \in I)$ called the {\it elementary crystals}.

\vskip 2mm

\begin{example} \label{ex:Bi} \hfill

\vskip 2mm

{\rm

(a) If $i \in I^{\text{re}}$, let $B_{i} = \{ (l) \mid l \ge 0 \}$ and define
\begin{equation*}
\begin{aligned}
& \text{wt}((l)) = - l \alpha_{i}, \\
& \epsilon_i((l)) =l, \ \ \phi_i((l)) = -l, \ \ \epsilon_j((l))= \phi_j((l)) = -\infty \ \ \text{for} \ j \neq i, \\
& \widetilde{e_i} ((l)) = (l-1), \ \ \widetilde{f_i}((l)) = (l+1), \\
& \widetilde{e}_{jk} ((l)) = \widetilde{f}_{jk}((l)) = 0 \ \ \text{for} \ j \neq i.
\end{aligned}
\end{equation*}
Then $B_{i}$ is an abstract crystal. The crystal graph of $B_{i}$ is given in Figure 1.

\vskip 2mm

{\bf Figure 1}
$$\xymatrix{(0)\ar[r]^i&(1)\ar[r]^i&(2)\ar[r]^i & \cdots}$$
We understand $(l)=0$ for $l<0$.
\vskip 2mm

(b) For $i \in I^{\text{im}} \setminus I^{\text{iso}}$, let $B_{i} = \{ \mathbf{c} = (c_1, \ldots, c_r) \mid r \ge 0, \  \text{$\mathbf{c}$ is a composition} \}$ and define
\begin{equation*}
\begin{aligned}
& \text{wt}(\mathbf{c}) = - |\mathbf{c}| \alpha_{i}, \ \ \epsilon_i(\mathbf{c}) = 0, \ \ \phi_i(\mathbf{c}) = - |\mathbf{c}| a_{ii}, \\
& \epsilon_j(\mathbf{c}) = \phi_j(\mathbf{c}) = - \infty\ \ \text{for} \ j \neq i, \\
& \widetilde{e}_{il}(\mathbf{c}) = \begin{cases} \mathbf{c} \setminus l = (c_2, \ldots, c_r)  & \text{if} \ c_1 = l. \\
0   & \text{otherwise},   \end{cases} \\
& \widetilde{f}_{il}(\mathbf{c}) = (l, \mathbf{c}) = (l, c_1, \ldots, c_r), \\
& \widetilde{e}_{jk}(\mathbf{c}) = \widetilde{f}_{jk}(\mathbf{c}) = 0 \ \ \text{for}  \ j \neq i.
\end{aligned}
\end{equation*}
Then $B_{i}$ is an abstract crystal. The crystal graph of $B_{i}$ is given in Figure 2.

\vskip 2mm

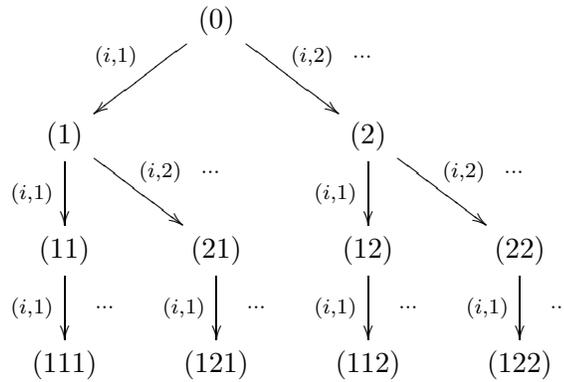
{\bf Figure 2}
$$\xymatrix{&(0)\ar[dl]_{(i,1)}\ar[dr]^{(i,2) \ \ \cdots}&&\\
(1)\ar[d]_{(i,1)}\ar[dr]^{(i,2) \ \ \cdots}&&(2)\ar[d]_{(i,1)}\ar[dr]^{(i,2) \ \ \cdots}&\\
(11)\ar[d]_{(i,1)}^{\ \ \cdots}&(21)\ar[d]_{(i,1)}^{\ \ \cdots}&(12)\ar[d]_{(i,1)}^{\ \ \cdots}&(22)\ar[d]_{(i,1)}^ {\ \ \cdots}\\
(111)&(121)&(112)&(122)}$$


\vskip 2mm

(c) For $i \in I^{\text{iso}}$, let $B_{i} = \{ \mathbf{c} = (c_1, \ldots, c_r) \mid r \ge 0, \  \text{$\mathbf{c}$ is a partition} \}$ and define

\begin{equation*}
\begin{aligned}
& \text{wt}(\mathbf{c}) = - |\mathbf{c}| \alpha_{i}, \ \ \epsilon_i(\mathbf{c}) = \phi_i(\mathbf{c}) = 0, \\
& \epsilon_j(\mathbf{c}) = \phi_j(\mathbf{c}) = - \infty\ \ \text{for} \ j \neq i, \\
& \widetilde{e}_{il}(\mathbf{c}) = \begin{cases} \mathbf{c} \setminus \{l\}  & \text{if $l$ is a part of $\mathbf{c}$}, \\
0   & \text{otherwise},   \end{cases} \\
& \widetilde{f}_{il}(\mathbf{c}) = \mathbf{c} \cup l =(l, c_1, \ldots, c_r ), \\
& \widetilde{e}_{jk}(\mathbf{c}) = \widetilde{f}_{jk}(\mathbf{c}) = 0 \ \ \text{for} \ j \neq i.
\end{aligned}
\end{equation*}
Then $B_{i}$ is an abstract crystal. A part of the crystal graph of $B_{i}$ is given in Figure 3.
\vskip 2mm

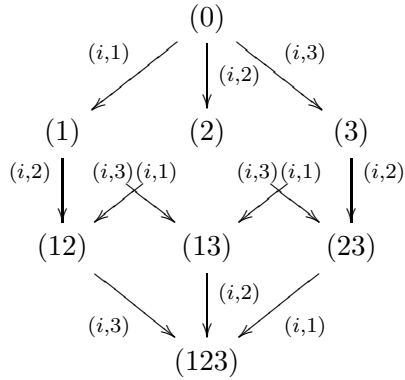
{\bf Figure 3}
$$\xymatrix{& (0)\ar[dl]_{(i,1)}\ar[d]^{(i,2)}\ar[dr]^{(i,3)}&\\
(1)\ar[d]_(.35){(i,2)}\ar[dr]|(.35){(i,3)}&(2)\ar[dl]|(.35){(i,1)}\ar[dr]|(.35){(i,3)}&(3)\ar[dl]|(.35){(i,1)}\ar[d]^(.35){(i,2)}\\
(12)\ar[dr]_{(i,3)}&(13)\ar[d]^(.4){(i,2)}&(23)\ar[dl]^{(i,1)}\\
& (123) & }$$
We will write $\mathbf{c}_{i}$ for $\mathbf{c} \in B_{i}$ when we would like to emphasize the index $i$.
}
\end{example}


\begin{definition} \label{def:normal}
An abstract crystal $B$ is {\it normal} if
$$\epsilon_i(b)=0, \ \ \phi_i(b) \ge 0 \ \ \text{for all} \ i \in I^{\text{im}},  b \in B.$$
\end{definition}

\begin{remark}
The crystals $B(\infty)$, $B(\lambda)$ $(\lambda \in P^{+})$ and $C$ are normal, while $T_{\lambda}$ $(\lambda \in P)$ and $B_{i}$ $(i \in I)$ are not normal.
\end{remark}

\vskip 10mm

\section{Tensor product of crystals}

\vskip 2mm

Let $B_1$, $B_2$ be abstract crystals and let $B_1 \otimes B_2 = \{ b_1 \otimes b_2 \mid b_1 \in B_1, b_2 \in B_2 \}$ (as a set).  Define the maps $\text{wt}$, $\epsilon_i$, $\phi_i$ $(i \in I)$, $\widetilde{e}_{il}$, $\widetilde{f}_{il}$ $((i,l) \in I^{\infty})$ as follows.

\begin{equation} \label{eq:wt}
\begin{aligned}
& \text{wt}(b_1\otimes b_2)=\text{wt}(b_1)+\text{wt}(b_2),\\
& \epsilon_i(b_1\otimes b_2)=\text{max}(\epsilon_i(b_1),\epsilon_i(b_2)- \langle h_i, \text{wt} (b_1) \rangle),\\
& \phi_i(b_1\otimes b_2)=\text{max}(\phi_i(b_1)+\langle h_i, \text{wt}(b_2) \rangle,\phi_i(b_2)),\\
\end{aligned}
\end{equation}

If $i \in I^{\text{re}}$,
\begin{equation} \label{eq:real}
\begin{aligned}
& \widetilde{e}_{i}(b_1 \otimes b_2) = \begin{cases} \widetilde{e}_{i} b_1 \otimes b_2  & \text{if} \  \phi_{i}(b_1) \ge \epsilon_{i}(b_2), \\
b_1 \otimes \widetilde{e}_{i} b_2  & \text{if} \ \phi_{i}(b_1) < \epsilon_{i}(b_2),
\end{cases} \\
& \widetilde{f}_{i}(b_1 \otimes b_2) = \begin{cases} \widetilde{f}_{i} b_1 \otimes b_2  & \text{if} \ \phi_{i}(b_1) > \epsilon_{i}(b_2), \\
b_1 \otimes \widetilde{f}_{i} b_2  & \text{if} \ \phi_{i}(b_1) \le \epsilon_{i}(b_2).
\end{cases}
\end{aligned}
\end{equation}

If $i \in I^{\text{im}}$,
\begin{equation} \label{eq:im}
\begin{aligned}
& \widetilde{e}_{il}(b_1 \otimes b_2) =  \begin{cases} \widetilde{e}_{il} b_1 \otimes b_2  & \text{if} \ \phi_{i}(b_1) > \epsilon_{i}(b_2) - l a_{ii}, \\
0  & \text{if} \ \epsilon_{i}(b_{2}) < \phi_{i}(b_{1}) \le \epsilon_{i}(b_{2}) - l a_{ii}, \\
b_{1} \otimes \widetilde{e}_{il} b_2 \ \ & \text{if} \ \phi_{i}(b_1) \le \epsilon_{i}(b_{2}),
\end{cases}\\
& \widetilde{f}_{il}(b_1 \otimes  b_2) =\begin{cases} \widetilde{f}_{il} b_1 \otimes b_2  & \text{if} \ \phi_{i}(b_1) > \epsilon_{i}(b_2), \\
b_{1} \otimes \widetilde{f}_{il} b_2  & \text{if} \ \phi_{i}(b_1) \le \epsilon_{i}(b_2).
\end{cases}
\end{aligned}
\end{equation}

Note that if $a_{ii}=0$, we have $\widetilde{e}_{il}(b_1 \otimes b_2)=\widetilde{e}_{il} b_1 \otimes b_2 $ when $\ \phi_{i}(b_1) > \epsilon_{i}(b_2)$, and $\widetilde{e}_{il}(b_1 \otimes b_2)=b_{1} \otimes \widetilde{e}_{il} b_2 $ when $ \phi_{i}(b_1) \le \epsilon_{i}(b_{2})$.

\vskip 2mm

\begin{proposition}
{\rm
The set $B_{1} \otimes B_{2}$  together with the maps $\text{wt}$, $\epsilon_i$, $\phi_i$ $(i \in I)$, $\widetilde{e}_{il}$, $\widetilde{f}_{il}$ $((i,l) \in I^{\infty})$ is an abstract crystal. }
\end{proposition}

\begin{proof}  \ It is clear that the conditions (i), (ii)  and (iii) in Definition \ref{ac} hold. If $i \in I^{\text{re}}$, it was shown in \cite{Kas93} that all the conditions are satisfied.  For the condition (vii), if $\phi_{i}(b_1\otimes b_2)=-\infty$, since
$$\phi_i(b_1\otimes b_2)=\text{max}(\phi_i(b_1)+\langle h_i, \text{wt}(b_2) \rangle ,\phi_i(b_2)),$$ we have $\phi_{i}(b_1) = \phi_{i}(b_2) = -\infty$. Thus
$$\widetilde{e}_{il} (b_1 \otimes b_2)= \widetilde{f}_{il}(b_1 \otimes b_2) =0.$$

From now on, we will assume that  $i \in I^{\text{im}}$, $l \ge 1$. In this case, we don't have to check the condition (v). Hence we have only to verify the conditions (iv) and (vi).

\vskip 2mm






Suppose
$\widetilde{f}_{il}(b_1 \otimes b_2) = b_{1}' \otimes b_{2}'$.

\vskip 2mm

If $\phi_{i}(b_1) > \epsilon_{i}(b_2)$, then $\widetilde{f}_{il}(b_1 \otimes b_2) = \widetilde{f}_{il} b_1 \otimes b_2 = b_{1}' \otimes b_{2}'$. Thus
$$ \phi_{i}(b_{1}') = \phi_{i}(\widetilde{f}_{il}b_1) = \phi_{i}(b_1) - l a_{ii} > \epsilon_{i}(b_2) - l a_{ii}  =  \epsilon_{i}(b_{2}') - l a_{ii},$$ which implies $\widetilde{e}_{il}(b_{1}' \otimes b_{2}') = \widetilde{e}_{il} b_{1}' \otimes b_{2}' = b_{1} \otimes b_{2}.$

\vskip 2mm

If, $\phi_{i}(b_1) \le \epsilon_{i}(b_2)$, then $\widetilde{f}_{il}(b_1 \otimes b_2) = b_1 \otimes \widetilde{f}_{il} b_2 =b_{1}' \otimes b_{2}'$. Note that
$$\phi_{i}(b_{1}') = \phi_{i}(b_{1}) \le \epsilon_{i}(b_{2}) =  \epsilon_{i}(\widetilde{f}_{il} b_{2}) = \epsilon_{i}(b_{2}').$$
Thus $\widetilde{e}_{il}(b_{1}'\otimes b_{2}') = b_{1}' \otimes \widetilde{e}_{il} b_{2}' =b_1 \otimes b_2$.

\vskip 2mm

Conversely, suppose $\widetilde{e}_{il}(b_{1}' \otimes b_{2}') =b_1 \otimes b_2$.
Since  $\widetilde{e}_{il}(b_{1}' \otimes b_{2}') \neq 0$, we don't have to consider the case $\epsilon_{i}(b_{2}') < \phi_{i}(b_{1}') \le \epsilon_{i}(b_{2}') - l a_{ii}$.

\vskip 2mm

If $\phi_{i}(b_{1}') > \epsilon_{i}(b_{2}')  - l a_{ii}$, we have $\widetilde{e}_{il}(b_{1}' \otimes b_{2}') = \widetilde{e}_{il} b_{1}' \otimes b_{2}'.$ Note that
$$\phi_{i}(b_{1}) =\phi_{i}(\widetilde{e}_{il} b_{1}') = \phi_{i}(b_{1}') + l a_{ii} > \epsilon_{i}(b_{2}') = \epsilon_{i}(b_{2}).$$ Hence $\phi_{i}(b_{1}) > \epsilon_{i}(b_{2})$ and we obtain $\widetilde{f}_{il}(b_{1} \otimes b_{2}) = \widetilde{f}_{il} b_{1} \otimes b_{2} = b_{1}' \otimes b_{2}'$.

\vskip 2mm

If $\phi_{i}(b_{1}') \le \epsilon_{i}(b_{2}')$, then $\widetilde{e}_{il}(b_{1}' \otimes b_{2}') = b_{1}' \otimes \widetilde{e}_{il} b_{2}'$ and we have
$$\phi_{i}(b_{1}) = \phi_{i}(b_{1}') \le  \epsilon_{i}(b_{2}') =  \epsilon_{i}(\widetilde{e}_{il} b_{2}') = \epsilon_{i}(b_{2}),$$
which gives $\widetilde{f}_{il}(b_{1} \otimes b_{2}) = b_{1} \otimes \widetilde{f}_{il} b_{2} = b_{1}' \otimes b_{2}'$. Hence the condition (iv) is verified.

\vskip 2mm

To verify the condition (vi), let $b_{1} \otimes b_{2} \in B_{1} \otimes B_{2}$ such that $\widetilde{e}_{il}(b_{1} \otimes b_{2}) \neq 0$ for all $l \ge 1$.

\vskip 2mm

If $\phi_{i}(b_1) > \epsilon_{i}(b_2) - l a_{ii}$, then $\widetilde{e}_{il}(b_1 \otimes b_2) = \widetilde{e}_{il} b_{1} \otimes b_{2}$ and hence we get
\begin{equation*}
\begin{aligned}
\epsilon_{i}(\widetilde{e}_{il}(b_1 & \otimes b_2)) = \epsilon_{i}(\widetilde{e}_{il} b_{1} \otimes b_{2}) \\
&= \text{max}(\epsilon_{i}(\widetilde{e}_{il} b_{1}), \epsilon_{i}(b_{2}) - \langle h_{i}, \text{wt}(\widetilde{e}_{il}b_{1}) \rangle
) \\
&= \text{max} (\epsilon_{i}(b_1), \epsilon_{i}(b_2) - \langle h_i, \text{wt}(b_1) + l \alpha_{i} \rangle) \\
& = \text{max} (\epsilon_{i}(b_1), \epsilon(b_2) - \langle h_i, \text{wt}(b_1)\rangle - l a_{ii}) \\
& 
= \epsilon_{i} (b_1)= \epsilon_{i}(b_1 \otimes b_2),
\end{aligned}
\end{equation*}
\begin{equation*}
\begin{aligned}
\phi_{i}(\widetilde{e}_{il}(b_1  & \otimes b_2)) = \phi_{i}(\widetilde{e}_{il} b_1 \otimes b_2) \\
& = \text{max}(\phi_{i}(\widetilde{e}_{il} b_1) + \langle h_{i}, \text{wt}(b_2) \rangle, \phi_{i}(b_2)) \\
& = \text{max} (\phi_{i}(b_1) + l a_{ii} + \phi_{i}(b_2) - \epsilon_{i}(b_2), \phi_{i}(b_2)) \\
& = \phi_{i}(b_1)  + l a_{ii} + \phi_{i}(b_2) - \epsilon_{i}(b_2) = \phi_{i}(b_{1} \otimes b_{2})+la_{ii}.
\end{aligned}
\end{equation*}


If $\phi_{i}(b_{1}) \le \epsilon_{i}(b_2)$, then $\widetilde{e}_{il}(b_1 \otimes b_2) = b_1 \otimes \widetilde{e}_{il}b_2$ and hence we get
\begin{equation*}
\begin{aligned}
\epsilon_{i}(\widetilde{e}_{il} (b_1 & \otimes b_2)) =\epsilon_{i}(b_1 \otimes \widetilde{e}_{il} b_2) \\
& = \text{max} (\epsilon_{i}(b_{1}), \epsilon_{i}(\widetilde{e}_{il}b_2) - \langle h_i, \text{wt}(b_1) \rangle)\\
& = \text{max} (\epsilon_{i}(b_{1}), \epsilon_{i}(b_2) - \langle h_i, \text{wt}(b_1) \rangle)
= \epsilon_{i}(b_1 \otimes b_2),\\
\phi_{i}(\widetilde{e}_{il}(b_1  & \otimes b_2)) = \phi_{i}( b_1 \otimes \widetilde{e}_{il}b_2) \\
& = \text{max}(\phi_{i}( b_1) + \langle h_{i}, \text{wt}(\widetilde{e}_{il} b_2) \rangle, \phi_{i}(\widetilde e_{il}b_2)) \\
& = \text{max}(\phi_{i}(b_1) + \langle h_{i}, \text{wt}(b_2) \rangle + l a_{ii}, \phi_{i}(b_2)+la_{ii}) \\
& = \text{max}(\phi_{i}(b_1) + \langle h_{i}, \text{wt}(b_2) \rangle, \phi_{i}(b_2))+la_{ii}  \\
& = \phi_{i}(b_1 \otimes b_2)+la_{ii}.
\end{aligned}
\end{equation*}

\vskip 2mm

By a similar argument, we can verify the condition (vi) for the Kashiwara operators $\widetilde{f}_{il}$'s.
\end{proof}

\vskip 2mm

\begin{corollary}
{\rm If $B_1$ and $B_2$ are normal crystals, then $B_1 \otimes B_2$ is also a normal crystal.
}
\end{corollary}

\begin{proof} \ For $i \in I^{\text{im}}$, it is easy to see that
\begin{equation*}
\begin{aligned}
\epsilon_{i}(b_1 \otimes b_2) & = \text{max}(\epsilon_{i}(b_1), \epsilon_{i}(b_2) - \langle h_i, \text{wt}(b_{1}) \rangle )\\
& = \text{max}(0, -\phi_{i}(b_1)) = 0, \\
\phi_{i}(b_1 \otimes b_2) & = \text{max}(\phi_{i}(b_1) + \langle h_i, \text{wt}(b_2) \rangle , \phi_{i}(b_2)) \\
& = \text{max}(\phi_{i}(b_1) + \phi_{i}(b_2), \phi_{i}(b_2))  \ge 0,
\end{aligned}
\end{equation*}
as desired.
\end{proof}

\vskip 2mm

\begin{corollary} \label{cor:lambda-mu}
{\rm For $\lambda, \mu \in P^{+}$, there exists a unique strict crystal embedding
\begin{equation} \label{eq:lambda-mu}
\Phi_{\lambda, \mu}\colon B(\lambda + \mu) \rightarrow B(\lambda) \otimes B(\mu) \ \ \text{given by} \ v_{\lambda+\mu} \mapsto \ v_{\lambda} \otimes v_{\mu}.
\end{equation}
}
\end{corollary}

\begin{proof} \ By the tensor product rule, $v_{\lambda} \otimes v_{\mu}$ is the only element of weight $\lambda + \mu$ in $B(\lambda) \otimes B(\mu)$. Moreover, it is annihilated by all $\widetilde{e}_{il}$'s.  Hence the connected component of $B(\lambda) \otimes B(\mu)$ containing $v_{\lambda} \otimes v_{\mu}$ is the full subcrystal of $B(\lambda) \otimes B(\mu)$ which is isomorphic to $B(\lambda + \mu)$.
\end{proof}

\vskip 2mm

In the following proposition, we will show that the tensor product of crystals satisfies the {\it associativity law}.

\vskip 2mm

\begin{proposition} \label{prop:associative}

{\rm \ Let $B_{i}$ $(i=1, 2, 3)$ be abstract crystals. Then there exists a unique crystal isomorphism
\begin{equation} \label{eq:associative}
\Psi\colon (B_1 \otimes B_2) \otimes B_3 \overset{\sim} \longrightarrow B_1 \otimes (B_2 \otimes B_3)
\end{equation}
given by $(b_1 \otimes b_2) \otimes b_3) \mapsto b_1 \otimes (b_2 \otimes b_3)$, where  $b_i \in B_i$ $(i=1, 2, 3)$.
}
\end{proposition}

\begin{proof} \ It is clear that $\Psi$ is a bijection and preserves the function $\text{wt}$. We shall show that $\Psi$ preserves $\epsilon_{i}$, $\phi_{i}$ $(i \in I)$ and commutes with $\widetilde{f}_{il}$ $((i,l) \in I^{\infty})$.  As we have seen in (\ref{rem:iso}), since $\Psi$ is a bijection, we don't have to check the commutativity with $\widetilde{e}_{il}$'s.

\vskip 2mm

Let $b=(b_1 \otimes b_2) \otimes b_3$ and $b' = \Psi(b) =b_1 \otimes (b_2 \otimes b_3)$. Our proof will be divided into the following cases.

\vskip 2mm

\noindent
{\bf Case 1:} $\phi_{i}(b_1 \otimes b_2) > \epsilon_{i}(b_3)$.

\vskip 2mm

(1) First, assume that $\phi_{i}(b_1) > \epsilon_{i}(b_2)$. In this case, we have
$$\phi_{i}(b_1 \otimes b_2) = \phi_{i}(b_1) + \phi_{i}(b_2) -\epsilon_{i}(b_2) > \epsilon(b_3), \ \ \epsilon_{i}(b_1 \otimes b_2) = \epsilon_{i}(b_1), $$ which yields
\begin{equation}\label{eq:a}
\phi_{i}(b_1) + \phi_{i}(b_2) > \epsilon_{i}(b_2) + \epsilon_{i}(b_3).
\end{equation}
Thus we obtain
\begin{equation*}
\begin{aligned}
\phi_{i}(b) & = \phi_{i}((b_1 \otimes b_2) \otimes b_3)
 = \phi_{i}(b_1 \otimes b_2) + \phi_{i}(b_3) - \epsilon_{i}(b_3)\\
& = \phi_{i}(b_1) + \phi_{i}(b_2) + \phi_{i} (b_3) - \epsilon_{i}(b_2) - \epsilon_{i}(b_3), \\
\epsilon_{i}(b) & = \epsilon_{i}((b_1 \otimes b_2) \otimes b_3) = \epsilon_{i}(b_1 \otimes b_2) = \epsilon_{i}(b_1) , \\
\widetilde{f}_{il}(b) & = \widetilde{f}_{il}((b_{1} \otimes b_{2}) \otimes b_3) = \widetilde{f}_{il}(b_1 \otimes b_2) \otimes b_3 = (\widetilde{f}_{il}b_1 \otimes b_2) \otimes b_3.
\end{aligned}
\end{equation*}

On the other hand, to deal with $b'$, we compare $\phi_{i}(b_1)$ and $\epsilon_{i}(b_2 \otimes b_3)$
and obtain

\begin{equation*}
\begin{aligned}
\phi_{i}(b') & = \phi_{i}(b_1 \otimes (b_2 \otimes b_3))\\
&= \phi_{i}(b_1) + \phi_{i}(b_2) + \phi_{i}(b_3) - \epsilon_{i}(b_2) -\epsilon_{i}(b_3),\\
\epsilon_{i}(b') & = \epsilon_{i}(b_1 \otimes (b_2 \otimes b_3)) =\epsilon_{i}(b_1), \\
\widetilde{f}_{il} (b') & = \widetilde{f}_{il}(b_1 \otimes (b_2 \otimes b_3)) =    \widetilde{f}_{il} b_1 \otimes (b_2 \otimes b_3).
\end{aligned}
\end{equation*}

\vskip 2mm

(2) Next, if $\phi_{i}(b_1) \le \epsilon_{i}(b_2)$, by a similar calculation, we obtain
\begin{equation*}
\begin{aligned}
\phi_{i}(b) &= \phi_{i}((b_1 \otimes b_2) \otimes b_3) = \phi_{i}(b_2) + \phi_{i}(b_3) - \epsilon_{i}(b_3), \\
\epsilon_{i}(b) & = \epsilon_{i}((b_1 \otimes b_2) \otimes b_3) = \epsilon_{i}(b_1) + \epsilon_{i}(b_2) - \phi_{i}(b_1), \\
\widetilde{f}_{il}(b) & =\widetilde{f}_{il}((b_1 \otimes b_2) \otimes b_3) = (b_1 \otimes \widetilde{f}_{il} b_2) \otimes b_3,
\end{aligned}
\end{equation*}
and
\begin{equation*}
\begin{aligned}
\phi_{i}(b') &= \phi_{i}(b_1 \otimes (b_2 \otimes b_3)) = \phi_{i}(b_2) + \phi_{i}(b_3) - \epsilon_{i}(b_3), \\
\epsilon_{i}(b') & = \epsilon_{i}(b_1 \otimes (b_2 \otimes b_3)) = \epsilon_{i}(b_1) + \epsilon_{i}(b_2) - \phi_{i}(b_1), \\
\widetilde{f}_{il} (b') & = \widetilde{f}_{il}(b_1 \otimes (b_2 \otimes b_3))  =
 b_1 \otimes (\widetilde{f}_{il} b_2 \otimes b_3).
\end{aligned}
\end{equation*}

\vskip 2mm

\noindent
{\bf Case 2:} $\phi_{i}(b_1 \otimes b_2) \le \epsilon_{i}(b_3)$.

\vskip 2mm

In this case, by comparing $\phi_{i}(b_1)$ and $\epsilon_{i}(b_2 \otimes b_3)$, we obtain

\begin{equation*}
\begin{aligned}
\phi_{i}(b) & = \phi_{i}((b_1 \otimes b_2) \otimes b_3) = \phi_{i}(b_3), \\
\epsilon_{i}(b) & = \epsilon_{i}((b_1 \otimes b_2) \otimes b_3)=\epsilon_{i}(b_1) +\epsilon_{i}(b_2) + \epsilon_{i}(b_3) -\phi_{i}(b_1) - \phi_{i}(b_2), \\
\widetilde{f}_{il}(b) &  = \widetilde{f}_{il}((b_1 \otimes b_2) \otimes b_3) =(b_1 \otimes b_2) \otimes \widetilde{f}_{il} b_3,
\end{aligned}
\end{equation*}
and
\begin{equation*}
\begin{aligned}
\phi_{i}(b') & = \phi_{i}(b_1 \otimes (b_2 \otimes b_3)) = \phi_{i}(b_3), \\
\epsilon_{i}(b') & = \epsilon_{i}(b_1 \otimes (b_2 \otimes b_3)) =  \epsilon_{i}(b_1) + \epsilon_{i}(b_2) + \epsilon_{i}(b_3) - \phi_{i}(b_1) - \phi_{i}(b_2), \\
\widetilde{f}_{il}(b') & = \widetilde{f}_{il}(b_1 \otimes (b_2 \otimes b_3)) = b_1 \otimes (b_2 \otimes \widetilde{f}_{il} b_3).
\end{aligned}
\end{equation*}

Thus we have proved all of our assertions.
\end{proof}

\vspace{10pt}

\section{Crystal embedding theorem}

\vskip 2mm

In this section, we prove one of the main results in this paper, the {\it crystal embedding theorem} for quantum Borcherds-Bozec algebras.

\vskip 2mm

\begin{theorem} \label{thm:crystal_embedding}
{\rm

For $i \in I$, there is a unique strict crystal embedding
\begin{equation*}
\Psi_{i}: B(\infty) \hookrightarrow B(\infty) \otimes B_{i}
\ \ \text{given by} \ \ \mathbf{1} \mapsto \mathbf{1} \otimes (0)_{i}.
\end{equation*}
}
\end{theorem}

\begin{proof} \ Since there is only one vector of weight $0$ in $B(\infty) \otimes B_{i}$, which is $\mathbf{1} \otimes (0)_{i}$, $\Psi_{i}$ should send $\mathbf{1}$ to $\mathbf{1} \otimes (0)_{i}$ if $\Psi_{i}$ exists, because it is a crystal morphism.

\vskip 2mm

Let $b = \widetilde{f}_{i_1, l_1} \cdots \widetilde{f}_{i_r, l_r} \, \mathbf{1} \in B(\infty)$. Choose $\lambda \in P^{+}$ such that
\begin{itemize}
\item[(i)] $\langle h_j, \lambda \rangle \gg 0$ for all $j \in I$,

\item[(ii)] \, $b \in \text{Im} \pi_{\lambda}$,
\end{itemize}
where $\pi_{\lambda}: B(\lambda) \rightarrow B(\infty)$ is the injective map given in Proposition \ref{P1}.

\vskip 2mm

Set $b_{\lambda} =  \widetilde{f}_{i_1, l_1} \cdots \widetilde{f}_{i_r, l_r} \, v_{\lambda} \in B(\lambda)$ so that $\pi_{\lambda}(b_{\lambda}) = b \in B(\infty)$.

\vskip 2mm

Let $l = \langle h_i, \lambda \rangle $ and $\mu=\lambda - l \Lambda_i \in P^{+}$. Then $\langle h_i, \mu \rangle =0$ and by Corollary \ref{cor:lambda-mu} there exists a unique strict crystal embedding
$$\Phi_{\mu, l \Lambda_i} : B(\lambda) \rightarrow B(\mu) \otimes B(l \Lambda_i) \ \ \text{given by} \ \ v_{\lambda} \mapsto v_{\mu} \otimes v_{l \Lambda_{i}}.$$
We will show that
\begin{itemize}
\item[(1)]  $ \Phi_{\mu, l \Lambda_i}(b_{\lambda}) = b' \otimes \widetilde{f}_{i, \mathbf{c}} v_{l \Lambda_i}$ for some $b' \in B(\mu)$, $\mathbf{c} \in B_{i}$,
\item[(2)]  $\pi_{\mu}(b') \otimes \mathbf{c} \in B(\infty) \otimes B_{i}$ does not depend on the choice of $\lambda \gg 0$.
\end{itemize}
 Here, if $i\in I^{\text{re}}$, then $\mathbf c\in B_{i}$ is a non-negative integer $c$ and we understand $\widetilde{f}_{i, \mathbf{c}}=\widetilde{f}_i^c$.

\vskip 2mm

Once our claims are proved, we will get a well-defined map
$$\Psi_{i} : B(\infty) \rightarrow B(\infty) \otimes B_{i}  \ \ \text{given by} \ \ b \mapsto \pi_{\mu}(b') \otimes \mathbf{c}.$$

We will prove our assertions by induction on $r \ge 0$. When $r=0$, our assertion is obvious. Assume that $r>0$ and our assertion is true for $r-1$.

\vskip 2mm

Set  $b_{1} = \widetilde{f}_{i_2, l_2} \cdots \widetilde{f}_{i_r, l_r} v_{\lambda}$. By our induction hypothesis, we have
\begin{itemize}
\item[(1a)] $\Phi_{\mu, l \Lambda_{i}}(b_1) = b_{1}' \otimes \widetilde{f}_{i, {\mathbf c}'} v_{l \Lambda_i}$ for some $b_1' \in B(\mu)$, ${\mathbf c}' \in B_{i}$,
\item[(2a)] $\pi_{\mu}(b_1') \otimes {\mathbf c}' \in B(\infty) \otimes B_{i}$ does not depend on the choice of $\lambda$.
\end{itemize}

\vskip 2mm

Therefore it suffices to show that
\begin{itemize}
\item[(1b)] $\widetilde{f}_{i_1, l_1}(b_1' \otimes \widetilde{f}_{i, {\mathbf c}'} v_{l \Lambda_i}) = b' \otimes \widetilde{f}_{i, \mathbf{c}} v_{l \Lambda_i}$ for some $b' \in B(\mu)$, $\mathbf{c} \in B_{i}$.

\item[(2b)] $\widetilde{f}_{i_1, l_1}(\pi_{\mu}(b_1') \otimes {\mathbf c}') = \pi_{\mu}(b') \otimes \mathbf{c}$.
\end{itemize}

\vskip 2mm

If $i_1 = i$, then $\langle h_i, \mu \rangle =0$ and hence
\begin{equation*}
\begin{aligned}
& \phi_{i}(b_1') = \phi_{i}(\pi_{\mu}(b_1')), \\
& \epsilon_{i}(\widetilde{f}_{i, {\mathbf c}'} v_{l \Lambda_i}) = \epsilon_{i}({\mathbf c}')
= \begin{cases} c'  & \text{if} \ i \in I^{\text{re}}, \\
0 & \text{if} \ i \in I^{\text{im}}.
\end{cases}
\end{aligned}
\end{equation*}
Therefore, $\widetilde{f}_{i_1, l_1}$ acts on the 1st component (resp. 2nd component) of (1b) if and only if it acts on the 1st component (resp. 2nd component) of (2b), which proves our claim.

\vskip 2mm

If $i_1 \neq i$, then $\epsilon_{i_1}(\mathbf c')=-\infty$ and
$$\phi_{i_1}(b_1')=\phi_{i_1}(\pi_\mu(b_1'))+\mu(h_{i_1})\gg 0=\epsilon_{i_1}(\widetilde{f}_{i,\mathbf c'} v_{l\Lambda_i}).$$
Hence $\widetilde{f}_{i_1, l_1}$ acts on the 1st component of (1b) and (2b), which yields our claim.

\vskip 2mm

It is straightforward to verify that $\Psi_{i}$ is a strict crystal morphism.  \end{proof}

\vskip 2mm

Our next goal is to provide a characterization of the crystals $B(\infty)$ and $B(\lambda)$ $(\lambda \in P^{+})$. To this end, we introduce an important family of crystals arising from tensor products of elementary crystals.

\vskip 2mm

Let ${\mathbf i} = (i_1, i_2, \ldots)$ be an infinite sequence of indices in $I$ such that every $i \in I$ appears infinitely many times. Set
\begin{equation} \label{eq:Bseq}
B_{\mathbf{i}} = \{ b = \cdots \otimes {\mathbf c}_k \otimes \cdots \otimes {\mathbf c}_1 \mid {\mathbf c}_{k} \in B_{i_k}, \ {\mathbf c}_{k}=(0)_{i_k} \ \text{for} \ k \gg 0 \}.
\end{equation}

\vskip 2mm

Using the tensor product rule, there is a crystal structure on $B_{\mathbf{i}}$ defined as follows. (See \cite{JKKS2007} for a more rigorous and detailed treatment.)

\vskip 2mm

Let $b = \cdots \otimes {\mathbf c}_k \otimes \cdots \otimes {\mathbf c}_1$, where ${\mathbf c}_k \in B_{i_k}$. Then we define
\begin{equation*}
\begin{aligned}
\text{wt}(b) & = - \sum_{k \ge 1}|\mathbf{c}_{k}| \, \alpha_{i_k}, \\
\epsilon_{i}(b) & = \begin{cases}
\text{max} \{c_k + \sum_{p>k} {c}_p \, a_{i, i_p}  \mid k \ge 1, \ i_k =i \}  & \text{if} \ i \in I^{\text{re}}, \\
0  & \text{if} \ i \in I^{\text{im}},
\end{cases} \\
\phi_{i}(b) & = \begin{cases}
\text{max}\{-c_k - \sum_{1 \le p <k} c_p \, a_{i, i_p}  \mid k\ge 1, i_k=i \}  & \text{if} \ i \in I^{\text{re}}, \\
-\sum_{k \ge 1} |{\mathbf c}_{k}| a_{i, i_k}  & \text{if} \ i \in I^{\text{im}}.
\end{cases}
\end{aligned}
\end{equation*}

To define the Kashiwara operators, we first assume $i \in I^{\text{re}}$ and let $s$ (resp. $t$)  be the  largest (resp. smallest) integer $k \ge 1$ such that
\begin{itemize}
\item[(i)] $i_k = i$,
\item[(ii)] $c_k + \sum_{p>k} c_p\,a_{i, i_p} = \epsilon_i(b)$.
\end{itemize}
Then we define
\begin{equation*}
\begin{aligned}
& \widetilde{e}_{i}(b) = \begin{cases} \cdots \otimes \mathbf{c}_{s+1} \otimes (c_s-1) \otimes \mathbf{c}_{s-1}  \otimes \cdots \otimes {\mathbf c}_1 & \text{if} \ \epsilon_{i}(b) > 0, \\
0 & \text{otherwise},
\end{cases} \\
& \widetilde{f}_{i}(b) = \cdots \otimes \mathbf{c}_{t+1} \otimes (c_t +1) \otimes {\mathbf c}_{t-1}  \otimes \cdots \otimes {\mathbf c}_1.
\end{aligned}
\end{equation*}

Suppose $i \in I^{\text{im}}$ and let $r$ be the smallest integer $k \ge 1$ such that
\begin{itemize}
\item[(i)] $i_k = i$,
\item[(ii)] $\sum_{p>k} |{\mathbf c}_{p}| \,a_{i,i_p}  =0$.
\end{itemize}
Then we define
$$\widetilde{f}_{i,l}(b) = \cdots \otimes {\mathbf c}_{r+1} \otimes (l, {\mathbf c}_{r}) \otimes {\mathbf c}_{r-1} \otimes \cdots \otimes {\mathbf c}_{1}.$$

Assume further that
\begin{itemize}
\item[(i)] the 1st component of ${\mathbf c}_{r}$ is $l$ or ${\mathbf c}_{r}$ is a partition having $l$ as a part,
\item[(ii)] $\sum_{s<p\le r} |{\mathbf c}_{p}| \, a_{i, i_p} < l a_{ii}$ for any $s$ with $i_s =i$ and $1 \le s < r$.
\end{itemize}
In this case, we define
$$\widetilde{e}_{i,l}(b) =
\cdots \otimes {\mathbf c}_{r+1} \otimes ({\mathbf c}_{r} \setminus l) \otimes {\mathbf c}_{r-1} \otimes \cdots \otimes {\mathbf c}_{1}.$$

Otherwise, we define $\widetilde{e}_{il}(b)=0$.

\vskip 2mm

For each $N\ge 1$, along the sequence $\mathbf{i} =(i_1, i_2, \ldots)$, we apply the crystal embedding theorem repeatedly to get a strict crystal embedding
\begin{equation*}
\Psi^{(N)}\colon B(\infty) \hookrightarrow B(\infty) \otimes B_{i_1} \hookrightarrow B(\infty) \otimes B_{i_2} \otimes B_{i_1} \hookrightarrow \cdots \hookrightarrow B(\infty) \otimes B_{i_N} \otimes \cdots \otimes B_{i_1}.
\end{equation*}

For each $b \in B(\infty)$, it is easy to see that there exists some $N \ge 1$ satisfying
$$\Psi^{(N)}(b) = \mathbf{1} \otimes {\mathbf c}_{N}  \otimes \cdots \otimes {\mathbf c}_{1} \in B(\infty) \otimes B_{i_N} \otimes \cdots \otimes B_{i_1}.$$
Hence to each $b \in B(\infty)$, one can associate a unique element
$$b_{\mathbf{i}} =  \cdots \otimes (0)_{i_{N+1}} \otimes {\mathbf c}_{N}  \otimes \cdots \otimes {\mathbf c}_{1} \in B_{\mathbf{i}} ,$$
which yields a strict crystal embedding
$$\Psi_{\mathbf{i}}\colon B(\infty) \hookrightarrow B_{\mathbf{i}}.$$
Therefore $B(\infty)$ is isomorphic to the connected component of $B_{\mathbf{i}}$ containing the element $(0)_{\mathbf{i}} = \cdots \otimes (0)_{i_{N}} \otimes \cdots \otimes (0)_{i_1} $. In particular, $\mathbf{1}$ is mapped onto $(0)_{\mathbf{i}}$.

\vskip 2mm

We will now give a characterization of $B(\infty)$ as an application of the crystal embedding theorem.

\vskip 2mm

\begin{theorem} \label{thm:Binfty}

{\rm Let $B$ be a crystal satisfying the following conditions.

\begin{itemize}
\item[(i)] $\text{wt}(B) \subset Q^{-}$,
\item[(ii)] there exists an element $b_{0} \in B$ such that $\text{wt}(b_{0}) =0$,
\item[(iii)] for any $b \neq b_0$, there is some $(i, l) \in I^{\infty}$ such that $\widetilde{e}_{il}\, b \neq 0$,
\item[(iv)] for each $i \in I$, there exists a strict crystal embedding $\Psi_{i}\colon B \hookrightarrow B \otimes B_{i}$.
\end{itemize}

\vskip 2mm

Then there exists a crystal isomorphism
\begin{equation*}
B  \overset{\sim} \longrightarrow B(\infty)  \ \ \text{given by} \ \  b_{0} \mapsto  \mathbf{1}.
\end{equation*}
}
\end{theorem}

\begin{proof} \ Note that for any element $b \in B$ with $\text{wt}(b)=0$, we have $\widetilde{e}_{il}(b)=0$ for all $(i,l) \in I^{\infty}$, for  otherwise, $\text{wt}(\widetilde{e}_{il}\,b) = l \alpha_{i} \notin Q^{-}$.  Hence by the condition (iii), if $\text{wt}(b)=0$, then $b=b_0$,
which implies $b_{0}$ is the only element in $B$ such that  $\text{wt}(b)=0$. It follows that the crystal embedding $\Psi_{i}$ maps $b_{0}$ to $b_{0} \otimes (0)_{i}$.

\vskip 2mm

Take an infinite sequence $\mathbf{i} = (i_1, i_2, \ldots)$ such that every $i \in I$ appears infinitely many times. Then for each $N \ge 1$, we obtain a strict embedding
$$\Psi^{(N)}\colon B\hookrightarrow B \otimes B_{i_1} \hookrightarrow B \otimes B_{i_2} \otimes B_{i_1} \hookrightarrow \cdots \hookrightarrow B \otimes B_{i_N} \otimes \cdots \otimes B_{i_1}$$
sending $b_{0}$ to $b_{0} \otimes (0)_{i_N} \otimes \cdots \otimes (0)_{i_1}$.

\vskip 2mm

As we have seen in the discussion above the theorem, for each $b \in B$, there exists $N\ge 1$ such that
$$\Psi^{(N)}(b) = b_{0} \otimes {\mathbf c}_{N} \otimes \cdots \otimes {\mathbf c}_{1},$$
which yields a strict crystal embedding
$$B \hookrightarrow B(\mathbf{i}) \ \ \text{given by} \ \ b \mapsto \cdots \otimes (0)_{i_{N+1}}
\otimes {\mathbf c}_{N} \otimes \cdots \otimes {\mathbf c}_{1} \, ({\mathbf c}_{k} \in B_{i_k}).$$
Thus $B$ is isomorphic to the connected component of $B_{\mathbf{i}}$ containing the element $(0)_{\mathbf{i}} = \cdots \otimes (0)_{i_{N}} \otimes \cdots \otimes (0)_{i_1}$, which implies $B$ is isomorphic to $B(\infty)$ and $b_{0}$ is mapped to $\mathbf{1}$.
\end{proof}

\vskip 2mm

Let us turn to the crystal  $B(\lambda)$ $(\lambda \in P^{+})$. The properties of the crystals $T_{\lambda}=\{t_{\lambda}\}$ and $C = \{c\}$ introduced in Example \ref{ex:Tlambda} will play an important role in our characterization of $B(\lambda)$.

\vskip 2mm

For $\lambda \in P^{+}$, consider the injective map
$$\psi_{\lambda}\colon B(\lambda) \rightarrow B(\infty) \otimes T_{\lambda} \ \ \text{given by} \ \ b \mapsto \pi_{\lambda}(b) \otimes t_{\lambda},$$
where $\pi_{\lambda}\colon B(\lambda) \rightarrow B(\infty)$ is the injective map given in Proposition \ref{P1}. One can immediately see that $\psi_{\lambda}$ preserves the function $\text{wt}$. Since $\epsilon_{i}(t_{\lambda}) = \phi_{i}(t_{\lambda})= -\infty$, for $b \in B(\lambda)$, we have
\begin{equation*}
\begin{aligned}
\epsilon_{i}(b) & = \epsilon_{i}(\pi_{\lambda}(b))= \epsilon_{i}(\pi_{\lambda}(b) \otimes t_{\lambda}), \\
\phi_{i}(b) & = \epsilon_{i}(b) + \langle h_i, \text{wt}(b) \rangle \\
& = \epsilon_{i}(\pi_{\lambda}(b) \otimes t_{\lambda}) + \langle h_i, \text{wt}(\pi_{\lambda}(b) \otimes t_{\lambda}) \rangle \\
& =  \phi_{i}(\pi_{\lambda}(b) \otimes t_{\lambda}).
\end{aligned}
\end{equation*}
Hence $\psi_{\lambda}$ preserves $\epsilon_{i}$, $\phi_{i}$ $(i \in I)$.

\vskip 2mm

If $b \in B(\lambda)$ and $\widetilde{f}_{il}\, b \in B(\lambda)$, Proposition \ref{P1} yields
\begin{equation*}
\widetilde{f}_{il}(\psi_{\lambda}(b)) = \widetilde{f}_{il}(\pi_{\lambda}(b) \otimes t_{\lambda})
= \widetilde{f}_{il}(\pi_{\lambda}(b)) \otimes t_{\lambda}
 = \pi_{\lambda}(\widetilde{f}_{il}\, b) \otimes t_{\lambda} = \psi_{\lambda}(\widetilde{f}_{il}\, b).
\end{equation*}
Moreover,
\begin{equation*}
\widetilde{e}_{il}(\psi_{\lambda}(b)) = \widetilde{e}_{il}(\pi_{\lambda}(b) \otimes t_{\lambda}) = \widetilde{e}_{il} (\pi_{\lambda}(b)) \otimes t_{\lambda} = \pi_{\lambda} (\widetilde{e}_{il}\, b) \otimes t_{\lambda} =\psi_{\lambda}(\widetilde{e}_{il}\, b).
\end{equation*}
Thus $\psi_{\lambda}$ is a crystal morphism commuting with all $\widetilde{e}_{il}$'s.



\vskip 2mm

Consider the injective map
$$\iota_{\lambda}\colon B(\lambda) \rightarrow B(\infty) \otimes T_{\lambda} \otimes C \ \
\text{given by} \ b \mapsto
\psi_{\lambda}(b) \otimes c.$$

For $b \in B(\lambda)$, we have
\begin{equation*}
\begin{aligned}
\text{wt}(\iota_{\lambda}(b)) & = \text{wt}(\psi_{\lambda}(b)) + \text{wt}(c) = \text{wt}(b), \\
\phi_{i}(\iota_{\lambda}(b)) & = \text{max}(\phi_{i}(\psi_{\lambda}(b)) + \langle h_i, \text{wt}(c) \rangle, \phi_{i}(c)) \\
& = \text{max}(\phi_{i}(b), 0) = \phi_{i}(b), \\
\epsilon_{i}(\iota_{\lambda}(b))& = \phi_{i}(\iota_{\lambda}(b)) - \langle h_i,
\text{wt}(\iota_{\lambda}(b)) \rangle \\
& = \phi_{i}(b) - \langle h_i, \text{wt}(b) \rangle = \epsilon_{i}(b).
\end{aligned}
\end{equation*}

Furthermore, since $b \in B(\lambda)$, we have $\phi_{i}(\psi_{\lambda}(b)) = \phi_{i}(b) \ge 0 = \epsilon_{i}(c)$. Thus if $i \in I^{\text{re}}$, then
\begin{equation*}
\begin{aligned}
\widetilde{e}_{i}(\iota_{\lambda}(b)) &= \widetilde{e}_{i}(\psi_{\lambda}(b) \otimes c)
= \widetilde{e}_{i}(\psi_{\lambda}(b)) \otimes c \\
&=\psi_{\lambda} (\widetilde{e}_{i}\, b) \otimes c
= \iota_{\lambda} (\widetilde{e}_{i}\,b),\\
\widetilde{f}_{i}(\iota_{\lambda}(b)) &= \widetilde{f}_{i}(\psi_{\lambda}(b) \otimes c)\\
&=\begin{cases} \widetilde{f}_{i} \psi_{\lambda}(b) \otimes c
= \psi_{\lambda}(\widetilde{f}_{i} \,b) \otimes c \ \ & \text{if} \ \phi_{i}(b)>0, \\
0 \ \  & \text{if} \ \phi_{i}(b) = 0,
\end{cases} \\
& = \iota_{\lambda}(\widetilde{f}_{i}\, b).
\end{aligned}
\end{equation*}

\vskip 2mm

If $i \in I^{\text{im}}$ and $\phi_{i}(b) = \phi_{i}(\psi_{\lambda}(b)) > -l a_{ii}$, then we have
$$\widetilde{e}_{il}(\iota_{\lambda}(b)) = \widetilde{e}_{il}(\psi_{\lambda}(b) \otimes c)
= \widetilde{e}_{il}(\psi_{\lambda}(b)) \otimes c = \psi_{\lambda} (\widetilde{e}_{il}\, b) \otimes c
=\iota_{\lambda}(\widetilde{e}_{il}\, b).$$

\vskip 2mm

If $i \in I^{\text{im}}$ and $0 < \phi_{i}(b)= \phi_{i}(\psi_{\lambda}(b)) \le - l a_{ii}$, then
\begin{equation*}
 \widetilde{e}_{il} (\iota_{\lambda}(b)) =\widetilde{e}_{il}(\psi_{\lambda}(b) \otimes c) =0.
\end{equation*}

\vskip 2mm

On the other hand, since $\epsilon_{i}(b)=0$, our condition implies
$$0 < \phi_{i}(b) = \langle h_i, \text{wt}(b) \rangle \le - l a_{ii}.$$
By the definition of category ${\mathcal O}_{\text{int}}$ (\cite[Definition 5.1]{Kang2019b}), we have $\widetilde{e}_{il}(b) =0$, which implies
$$\widetilde{e}_{il}(\iota_{\lambda}(b) ) = \iota_{\lambda}(\widetilde{e}_{il}\, b)=0$$
as desired.

\vskip 2mm

If $i \in I^{\text{im}}$ and $\phi_{i}(b) = \phi_{i}(\psi_{\lambda}(b))=0$, then $\langle h_i, \text{wt}(b) \rangle=0$, which implies $\widetilde{e}_{il}(b)=0$ and
\begin{equation*}
\widetilde{e}_{il}(\iota_{\lambda}(b)) = \widetilde{e}_{il}(\psi_{\lambda}(b) \otimes c)
=\psi_{\lambda}(b)\otimes \widetilde{e}_{il}(c)=0
=  \psi_{\lambda}(\widetilde{e}_{il}(b)) \otimes c
= \iota_{\lambda}(\widetilde{e}_{il}(b)).
\end{equation*}

\vskip 2mm

For the operators $\widetilde{f}_{il}$,
if $i \in I^{\text{im}}$ and $\phi_{i}(b) =\phi_{i}(\psi_{\lambda}(b)) > 0$, then
$$\widetilde{f}_{il}(\iota_{\lambda}(b)) = \widetilde{f}_{il}(\psi_\lambda(b) \otimes c)
=\widetilde{f}_{il}(\psi_\lambda(b)) \otimes c =\psi_{\lambda} (\widetilde{f}_{il}(b)) \otimes c
=\iota_{\lambda}(\widetilde{f}_{il}\, b).$$
If $i \in I^{\text{im}}$ and $\phi_{i}(b) =0$, then $\widetilde{f}_{il}(b) = 0$ and hence
$$\widetilde{f}_{il} (\iota_{\lambda}(b)) = \widetilde{f}_{il}(\psi_{\lambda}(b) \otimes c)
=\psi_{\lambda}(b) \otimes \widetilde{f}_{il}\, c = 0 = \iota_{\lambda}(\widetilde{f}_{il}\, b)$$
as desired.

\vskip 2mm

Therefore, $\iota_{\lambda}\colon B(\lambda) \rightarrow B(\infty) \otimes T_{\lambda} \otimes C$ is a strict crystal embedding and obtain the following characterization of $B(\lambda)$.

\vskip 2mm

\begin{theorem} \label{thm:Blambda}
{\rm Let $\lambda \in P^{+}$ be a dominant integral weight. Then the crystal $B(\lambda)$ is isomorphic to the connected component of $B(\infty) \otimes T_{\lambda}\otimes C$ containing $\mathbf{1} \otimes  t_{\lambda} \otimes c$.
}
\end{theorem}

\begin{remark}
(a) As was mentioned at the end of \cite{Bozec2014c}, our characterization can be applied to reproduce Bozec's geometric realization of $B(\infty)$ and $B(\lambda)$ (cf. \cite{KKS2009, KKS2012}).

(b) Our constuction of $B(\infty)$ and $B(\lambda)$ can also be used directly to the theory of (dual) perfect bases for quantum Borcherds-Bozec algebras.

(c) Our theory will play a crucial role in the categorification of quantum Borcherds-Bozec algebras and their highest weight modules via Khovanov-Lauda-Rouquier algebras and their cyclotomic quotients.

\end{remark}

\end{document}